\documentclass{amsart}
\usepackage{amsfonts}
\usepackage{amsmath,amssymb}
\usepackage{amsthm}
\usepackage{amscd}
\usepackage{graphics}
\usepackage{graphicx}

\usepackage[pagewise]{lineno}
\theoremstyle{definition}{
\newtheorem{Def}{{\rm Definition}}
\newtheorem{Ex}{{\rm Example}}
\newtheorem{Rem}{{\rm Remark}}

}
\theoremstyle{plain}
{
\newtheorem{Cor}{Corollary}
\newtheorem{Prop}{Proposition}
\newtheorem{Thm}{Theorem}

}

\begin{document}
\title[New explicit fold maps on 7-dimensional manifolds]{Explicit fold maps on 7-dimensional closed and simply-connected spin manifolds of new classes}
\author{Naoki Kitazawa}
\address{Institute of Mathematics for Industry, Kyushu University, 744 Motooka, Nishi-ku Fukuoka 819-0395, Japan\\
 TEL (Office): +81-92-802-4402 \\
 FAX (Office): +81-92-802-4405 \\
}
\email{n-kitazawa@imi.kyushu-u.ac.jp}
\urladdr{https://naokikitazawa.github.io/NaokiKitazawa.html}
\maketitle
\begin{abstract}
Closed (and simply-connected) manifolds whose dimensions are greater than $4$ are classified via sophisticated algebraic and abstract theory such as surgery theory and homotopy theory. It is difficult to handle $3$ or $4$-dimensional closed manifolds in such ways. However, the latter work via geometric and constructive ways is not so difficult. The assumption that the dimensions are not high enables us to handle the manifolds via diagrams for example. It is difficult to study higher dimensional manifolds in these ways, although it is natural and important.

In the present paper, we present such studies via {\it fold} maps, which are higher dimensional variants of Morse functions. The author previously constructed fold maps on $7$-dimensional closed and simply-connected manifolds satisfying additional conditions on cohomology rings, including so-called {\it exotic} homotopy spheres. This paper concerns fold maps on such manifolds of a wider class.


\end{abstract}


\maketitle
\section{Introduction and fold maps.}
\label{sec:1}
Closed (and simply-connected) manifolds whose dimensions are greater than $4$ are classified via classical and sophisticated algebraic and abstract theory such as surgery theory, homotopy theory, for example.

It is difficult to handle $3$ or $4$-dimensional closed manifolds in such ways. However, this work is, in geometric and constructive ways, not so difficult. The fact that the dimensions are not so high enables us to handle the manifolds via diagrams such as Heegaard diagrams and Kirby's diagrams for example.

It is difficult to study higher dimensional manifolds in these ways, although it is natural and important.
This paper presents studies via Morse functions and {\it fold} maps, which are higher dimensional versions of Morse functions on $7$-dimensional closed and simply-connected manifolds into the $4$-dimensional Euclidean space.

\subsection{Notation on differentiable maps and bundles.}
\label{subsec:1.1}
Throughout this paper, manifolds and maps between manifolds are fundamental objects and they are smooth and of class $C^{\infty}$. Diffeomorphisms on manifolds are always assumed to be smooth. The {\it diffeomorphism group} of a manifold is the group of all diffeomorphism on the manifolds.

For each integer $k \geq 0$, ${\mathbb{R}}^k$ denotes the $k$-dimensional Euclidean space with a canonical differentiable structure and a standard Euclidean metric. For $x \in {\mathbb{R}}^k$, $||x|| \geq 0$ denotes the distance between the origin $0$ and $x$ or equivalently, the value of the Euclidean norm at $x$. $S^k:=\{x \in {\mathbb{R}}^{k+1} \mid ||x||=1.\}$ denotes the $k$-dimensional unit sphere and $D^k:=\{x \in {\mathbb{R}}^{k} \mid ||x|| \leq 1.\}$ denotes the $k$-dimensional unit disk for $k \geq 0$. If a smooth manifold is diffeomorphic to a unit sphere, then it is said to be a {\it standard} sphere.
A (smooth) {\it homotopy} sphere means a smooth manifold homeomorphic to a unit sphere and if it is not diffeomorphic to any unit sphere, then it is said to be an {\it exotic} sphere. If a smooth manifold is diffeomorphic to a unit disk, then it is said to be a {\it standard} disk.

For bundles whose fibers are manifolds, the structure groups are subgroups of the diffeomorphism groups or the bundles are {\it smooth} unless otherwise stated. A {\it linear} bundle is a smooth bundle whose fiber is regarded as a unit sphere or a unit disk and whose structure group acts linearly in a canonical way on the fiber.

A {\it singular} point $p \in X$ of a differentiable map $c:X \rightarrow Y$ is a point at which the rank of the differential
$dc$ of the map is smaller than both the dimensions of the manifold of the domain and the manifold of the target or ${\rm rank} \quad {dc}_p < \min\{\dim X,,\dim Y\}$ holds where ${dc}_p$ denotes the differential at $p$. The set $S(c)$ of all singular points is the {\it singular set} of $c$. We call $c(S(c))$ the {\it singular value set} of $c$. We call $Y-c(S(c))$ the {\it regular value set} of $c$. A {\it singular {\rm (}regular{\rm )} value} is a point in the singular (resp. regular) value set of $c$.

\subsection{Fold maps and investigating algebraic topological and differential topological properties of manifolds via fold maps on them.}
\label{subsec:1.2}
Let $m>n \geq 1$ be integers. A smooth map between an $m$-dimensional smooth manifold with no boundary into an
$n$-dimensional smooth manifold with no boundary is said to be a {\it fold} map if at each singular point $p$, the map is represented as
$$(x_1, \cdots, x_m) \mapsto (x_1,\cdots,x_{n-1},\sum_{k=n}^{m-i}{x_k}^2-\sum_{k=m-i+1}^{m}{x_k}^2)$$
for suitable coordinates and an integer $0 \leq i(p) \leq \frac{m-n+1}{2}$.
We have the following three fundamental properties for this map.
\begin{itemize}
\item For any singular point $p$, $i(p)$ is shown to be unique.
\item The set consisting of all singular points of a fixed index of the map is an ($n-1$)-dimensional closed and smooth submanifold of dimension $n-1$ with no boundary of the $m$-dimensional manifold and the restriction to the singular set is a smooth immersion.
\end{itemize}
{$i(p)$ is said to be the {\it index} of $p$.
In the present paper, we consider fold maps such that the immersions in the third property satisfy the following two conditions.
\begin{enumerate}
\item The preimage of a point in the manifold of the target consists of at most two points.
\item For a preimage consisting of exactly two points as before, let the preimage be denoted by $\{p_1,p_2\}$ and the value at these points by $q$ in the manifold $Y$ of the target. We consider the subspace of the tangent vector space at $q \in Y$ represented as the intersection of the images of the differentials of the immersions at $p_1$ and $p_2$ and the dimension is $\dim Y-2$.
\end{enumerate}
Moreover, for general smooth immersions, we consider such immersions unless otherwise stated.
\begin{Def}
A smooth immersion satisfying the two conditions is said to be {\it normal}. A point in the manifold of the target whose preimage consists of two points is called a {\it crossing}. We extend these notions for a family of (smooth) immersions in a canonical way.
\end{Def}
Fold maps have been important tools in studying geometric properties of manifolds in the branch of the singularity theory of differentiable maps and application to geometry of manifolds as Morse functions are in so-called Morse theory.
Studies related to this were essentially started in 1950s by Thom and Whitney (\cite{thom} and \cite{whitney}). These studies are on smooth maps on manifolds whose dimensions are larger than $2$ into the plane. After various studies, recently, as an important topic, Saeki, Sakuma and so on, have been studying fold maps satisfying appropriate conditions and manifolds admitting them in \cite{saeki}, \cite{saeki2}, \cite{saekisakuma}, \cite{sakuma}, \cite{sakuma2} among others. Several studies such as \cite{kitazawa}, \cite{kitazawa2}, \cite{kitazawa3} and \cite{kitazawa4} of the author are motivated by these studies. The author has been developing these studies mainly to handle wider classes of manifolds admitting explicit fold maps. As closely related studies, the author also have been studying the topologies of {\it Reeb spaces} of fold maps in \cite{kitazawa5}, \cite{kitazawa6}, \cite{kitazawa8} and \cite{kitazawa9} for example: the {\it Reeb space} of a fold map is the polyhedron whose dimension is equal to the dimension of the manifold of the target, defined as the quotient space of the space of the domain obtained by the following rule: two points in the space are equivalent if and only if they are in a same connected component of the preimage of a point. These objects inherit much information on invariants such as homology groups and cohomology rings for manifolds admitting the maps in considerable cases and fundamental and important tools in the studies.

\subsection{Explicit fold maps on $7$-dimensional closed and simply-connected manifolds of a class and the main purpose of the present paper.}
\label{subsec:1.3}
In the present paper, the {\it product} of cohomology classes $u_1$ and $u_2$ means the {\it cup product} $u_1 \cup u_2$ of them. Hereafter, let ${\mathbb{N}}_{\leq x}$ denote the set of all positive integers smaller than or equal
to $x \in \mathbb{R}$.
\begin{Thm}[\cite{kitazawa7}]
\label{thm:1}
Let $A$, $B$ and $C$ be free finitely generated commutative groups each of whose rank is $a$, $b$ and $c$, respectively. Let $\{a_j\}_{j=1}^{a}$, $\{b_j\}_{j=1}^{b}$ and $\{c_j\}_{j=1}^{c}$ be bases of $A$, $B$ and $C$, respectively.
For each integer $1 \leq i \leq b$, suppose that a sequence $\{a_{i,j}\}_{j=1}^{a}$ of integers is given.
Let $p \in B \oplus C$.
In this situation, there exist a $7$-dimensional closed and simply-connected spin manifold $M$ and a fold map
$f:M \rightarrow {\mathbb{R}^4}$ satisfying the following properties.
\begin{enumerate}
\item
\label{thm:1.1}
$H_2(M;\mathbb{Z})$ is isomorphic to $A \oplus B$, $H_4(M;\mathbb{Z})$ is isomorphic to $B \oplus C$, and $H_3(M;\mathbb{Z})$ is free.
\item
\label{thm:1.2}
There exist suitable isomorphisms ${\phi}_2:A \oplus B \rightarrow H_2(M;\mathbb{Z})$ and
${\phi}_4:B \oplus C \rightarrow H_4(M;\mathbb{Z})$ and we can define the duals
${a_j}^{\ast}$, ${b_{j,2}}^{\ast}$, ${b_{j,4}}^{\ast}$ and ${c_j}^{\ast}$ of ${\phi}_2((a_j,0))$, ${\phi}_2((0,b_j))$, ${\phi}_4((b_j,0))$
and ${\phi}_4((0,c_j))$, respectively, to obtain bases of $H^2(M;\mathbb{Z})$ and $H^4(M;\mathbb{Z})$ in a canonical way. We define the composition of ${\phi}_4$ with an isomorphism mapping the elements of the basis to their Poincar\'e duals, denoted by ${\phi}_{4,3}:B \oplus C \rightarrow H^3(M;\mathbb{Z})$ and that of ${\phi}_2$ with an isomorphism mapping the elements of the basis to their Poincar\'e duals, denoted by ${\phi}_{2,5}:A \oplus B \rightarrow H^5(M;\mathbb{Z})$. In this situation, the following properties on poducts hold.
\begin{enumerate}
\item
\label{thm:1.2.1}
\begin{enumerate}
\item The product of ${a_{j_1}}^{\ast}$ and ${a_{j_2}}^{\ast}$ and that of ${b_{j_1,2}}^{\ast}$ and ${b_{j_2,2}}^{\ast}$ vanish for any pair $(j_1,j_2) \in {{\mathbb{N}}_{\leq a}}^2$ and $(j_1,j_2) \in {{\mathbb{N}}_{\leq b}}^2$ respectively.
\item The product of ${a_{j_1}}^{\ast}$ and ${b_{j_2,2}}^{\ast}$ is $a_{j_2,j_1}{b_{j_2,4}}^{\ast} \in H^4(M;\mathbb{Z})$ for any pair $(j_1,j_2) \in {\mathbb{N}}_{\leq a} \times {\mathbb{N}}_{\leq b}$.
\end{enumerate}
\item
\label{thm:1.2.2}
\begin{enumerate}
\item The product of ${a_j}^{\ast}$ and ${\phi}_{4,3}((0,c))$ vanishes for any $j \in {\mathbb{N}}_{\leq a}$ and $c \in C$.
\item The product of ${b_{j,2}}^{\ast}$ and ${\phi}_{4,3}((0,c))$ vanishes for any $j \in {\mathbb{N}}_{\leq b}$ and $c \in C$.
\item The product of ${b_{j_1,2}}^{\ast}$ and ${\phi}_{4,3}((b_{j_2},0))$ vanishes for any pair $(j_1,j_2) \in {{\mathbb{N}}_{\leq b}}^2$ of distinct numbers.
\item The product of ${a_{j_1}}^{\ast}$ and ${\phi}_{4,3}((b_{j_2},0))$ is $a_{j_2,j_1}{\phi}_{2,5}((0,b_{j_2})) \in H^5(M;\mathbb{Z})$ for any pair $(j_1,j_2) \in {\mathbb{N}}_{\leq a} \times {\mathbb{N}}_{\leq b}$.
\item The product of ${b_{i,2}}^{\ast}$ and ${\phi}_{4,3}((b_{i},0))$ is
${\Sigma}_{j=1}^{a} a_{i,j}{\phi}_{2,5}((a_j,0)) \in H^5(M;\mathbb{Z})$ for any $i \in {\mathbb{N}}_{\leq b}$.
\end{enumerate}
\item
\label{thm:1.2.3}
\begin{enumerate}
\item The product of  ${a_j}^{\ast}$ and ${\phi}_{2,5}((a_j,0))$ yields a generator of $H^7(M;\mathbb{Z})$ for any $j \in {\mathbb{N}}_{\leq a}$.
\item The product of ${a_{j_1}}^{\ast}$ and ${\phi}_{2,5}((a_{j_2},0))$ vanishes for any pair $(j_1,j_2) \in {{\mathbb{N}}_{\leq a}}^2$ of distinct numbers.
\item The product of  ${a_j}^{\ast}$ and ${\phi}_{2,5}((0,b))$ vanishes for any $j \in {\mathbb{N}}_{\leq a}$ and $b \in B$.
\item The product of  ${b_{j,2}}^{\ast}$ and ${\phi}_{2,5}((a,0))$ vanishes for any $j \in {\mathbb{N}}_{\leq b}$ and $a \in A$.
\item The product of  ${b_{j,2}}^{\ast}$ and ${\phi}_{2,5}((0,b_j))$ yields a generator of $H^7(M;\mathbb{Z})$ for any $j \in {\mathbb{N}}_{\leq b}$.
\item The product of  ${b_{j_1,2}}^{\ast}$ and ${\phi}_{2,5}((0,b_{j_2}))$ vanishes for any pair $(j_1,j_2) \in {{\mathbb{N}}_{\leq b}}^2$ of distinct numbers.
\item The product of ${b_{j,4}}^{\ast}$ and ${\phi}_{4,3}((b_j,0))$ yields a generator of $H^7(M;\mathbb{Z})$.
\item The product of  ${b_{j_1,4}}^{\ast}$ and ${\phi}_{4,3}((b_{j_2},0))$ vanishes for any pair $(j_1,j_2) \in {{\mathbb{N}}_{\leq b}}^2$ of distinct numbers.
\item The product of  ${b_{j,4}}^{\ast}$ and ${\phi}_{4,3}((0,c))$ vanishes for any $j \in {\mathbb{N}}_{\leq b}$ and $c \in C$.
\item The product of  ${c_j}^{\ast}$ and ${\phi}_{4,3}((b,0))$ vanishes for any $j \in {\mathbb{N}}_{\leq c} $ and $b \in B$.
\item The product of ${c_j}^{\ast}$ and ${\phi}_{4,3}((0,c_j))$ yields a generator of $H^7(M;\mathbb{Z})$ for any $j \in {\mathbb{N}}_{\leq c}$.
\item The product of ${c_{j_1}}^{\ast}$ and ${\phi}_{4,3}((0,c_{j_2}))$ vanishes for any pair $(j_1,j_2) \in {{\mathbb{N}}_{\leq c}}^2$ of distinct numbers.
\end{enumerate}
\end{enumerate}
\item The 3rd and the 5th Stiefel-Whitney classes of $M$ vanish.
\item Let $d$ be the isomorphism $d:H_4(M;\mathbb{Z}) \rightarrow H^4(M;\mathbb{Z})$ mapping the elements of the basis to their duals. The first Pontryagin class of $M$ is $4d \circ {\phi}_4(p) \in H^4(M;\mathbb{Z})$. The 4th Stiefel-Whitney class of $M$ vanishes.
\item $f {\mid}_{S(f)}$ is an embedding.
\item The index of each singular point is always $0$ or $1$.
\item Preimages of regular values are empty or diffeomorphic to $S^3$ or $S^3 \sqcup S^3$.
\end{enumerate} 
\end{Thm}

\begin{Thm}[\cite{kitazawa2} for example.]
\label{thm:2}
Every $7$-dimensional homotopy sphere admits a fold map $f$ into ${\mathbb{R}}^4$ satisfying the following properties. 
\begin{enumerate}
\item $f {\mid}_{S(f)}$ is embedding and $f(S(f))=\{x \in {\mathbb{R}}^4 \mid ||x||=1,2,3\}$.
\item The index of each singular point is always $0$ or $1$.
\item For each connected component of the regular value set of $f$, the preimage of a regular value is, empty, diffeomorphic to $S^3$, diffeomorphic to $S^3 \sqcup S^3$ and diffeomorphic to $S^3 \sqcup S^3 \sqcup S^3$, respectively.
\end{enumerate}
\end{Thm}
We will add precise expositions on Theorem \ref{thm:2} in the next section.
The present paper shows construction of fold maps into the $4$-dimensional Euclidean space on $7$-dimensional closed and simply-connected manifolds of a wider class. We will connect the present study and closely related studies to understand these manifolds in geometric and constructive ways. Existence problems of fold maps have been solved in \cite{eliashberg} and \cite{eliashberg2} by Eliashberg and others. Construction of fold maps has been more difficult in considerable cases. Such construction produces a main topic of studies of the author.

As an additional remark, the class of $7$-dimensional closed and simply-connected manifolds still produces attractive topics.
\cite{kreck} and \cite{wang} are recent explicit studies on explicit algebraic topological understandings of $7$-dimensional closed and simply-connected manifolds.

\subsection{The content of the present paper.}
\label{subsec:1.4}
The organization of the paper is as the following.
In the next section, we remark on subsection \ref{subsec:1.3}: we explain that fold maps of suitable classes affect the differentiable structures of the manifolds admitting them. After that, we construct new fold maps on $7$-dimensional closed, simply-connected and spin manifolds. The class of manifolds contains the manifolds in Theorems \ref{thm:1} and \ref{thm:2} and this is one of the main theorems (Theorem \ref{thm:5}). Key methods are based on \cite{kitazawa5}, \cite{kitazawa6} and \cite{kitazawa8}. Note that these articles are mainly for Reeb spaces and investigating the homology groups and the cohomology rings of manifolds are different from investigating those of the Reeb spaces and more difficult.
The third section is for a remark on the integral cohomology rings of the manifolds in Theorem \ref{thm:5}.
The last section is devoted to appendices. We present Theorem \ref{thm:6} as a slight extension of Theorem \ref{thm:5}. Remark \ref{rem:2} is a short remark on existence and construction of fold maps.

\section{Construction of new family of fold maps on $7$-dimensional closed and simply-connected manifolds of a new class.}
\label{sec:2}
Throughout the present paper, $M$ is an $m$-dimensional closed and connected manifold, $n<m$ is a positive integer and $f:M \rightarrow {\mathbb{R}}^n$ is a smooth map unless otherwise stated.
\subsection{Additional explanations on subsection \ref{subsec:1.3}: fold maps of suitable classes affect the differentiable structures of the manifolds admitting them.}
A {\it special generic} map is a fold map such that the index of each singular point is $0$. The class of special generic maps contains the class of Morse functions on closed manifolds with exactly two singular points, which are central objects in Reeb's theorem, characterizing spheres topologically except the case where the manifold is $4$-dimensional: in this case, a $4$-dimensional standard sphere is characterized as this. It also contains canonical projections of unit spheres.
Example \ref{ex:2} also presents simplest special generic maps, extending the class of the special generic maps into Euclidean spaces whose dimensions are greater than $1$ just before.
For example, Saeki, Sakuma and Wrazidlo discovered that special generic maps restrict the topologies and the differentiable structures of manifolds admitting them in considerable cases. One of their studies revealed that $7$-dimensional exotic spheres admit no special generic maps into
${\mathbb{R}}^4$, ${\mathbb{R}}^5$ and ${\mathbb{R}}^6$ and that some of such manifolds admit no such maps into
${\mathbb{R}}^3$ for example. See \cite{saeki}, \cite{saekisakuma} and  \cite{wrazidlo} for example.
\begin{Thm}[\cite{kitazawa7}]
\label{thm:3}
In the situation of Theorem \ref{thm:1}, if at least one of the following two hold, then $M$ never admits a special generic map into ${\mathbb{R}}^4$.
\begin{enumerate}
\item $p \in B \oplus C$ is not zero.
\item In some sequence $\{a_{i,j}\}_{j=1}^{a}$, at least one non-zero number exists.
\end{enumerate}
\end{Thm}
\begin{Def}
\label{def:2}
For a fold map $f:M \rightarrow {\mathbb{R}}^4$ as in Theorem \ref{thm:2} where $f(S(f))={\sqcup}_{r=1}^l \{x \in {\mathbb{R}}^4 \mid ||x||=r.\}$ instead of the original condition and the preimage of a point in the exactly one connected component of the regular value set which is diffeomorphic to an open disc, is the disjoint union of $l$ copies of $S^3$ for an integer $l>0$: originally $l=3$. We also assume that for each connected component $C$ of the singular value set and its small closed tubular neighborhood $N(C)$, the composition of $f {\mid}_{f^{-1}(N(C))}:f^{-1}(N(C)) \rightarrow N(C)$ with a canonical projection to $C$ gives a trivial bundle over $C$. We say that such $f$ is an {\it $l$ normal round fold map with standard spheres}.
\end{Def}
We can show the following theorem thanks to Theorem \ref{thm:2} with \cite{eellskuiper}, \cite{milnor} and other related theory on ($7$-dimensional) smooth homotopy spheres. It is also a classical important fact that there exist $28$ types of oriented smooth homotopy spheres.
\begin{Thm}[\cite{kitazawa2} for example.]
\label{thm:4}
Every $7$-dimensional smooth homotopy sphere $M$ admits a $3$ normal round fold map with standard spheres
$f:M \rightarrow {\mathbb{R}}^4$ as in Theorem \ref{thm:2}. Moreover, for a $7$-dimensional homotopy sphere $M$, we have the following characterization.
\begin{enumerate}
\item $M$ admits a $1$ normal round fold map with standard spheres $f:M \rightarrow {\mathbb{R}}^4$ if and only if $M$ is a standard sphere.
\item $M$ admits a $2$ normal round fold map with standard spheres $f:M \rightarrow {\mathbb{R}}^4$ if and only if the homotopy sphere is one of 16 types of the 28 types {\rm (}oriented homotopy spheres of these exactly $16$ types are represented as total spaces of linear $S^3$-bundles over $S^4$ and a standard sphere is one of these 16 types{\rm )}.
\end{enumerate}
\end{Thm}
In a class which is not the class of special generic maps, the following fact was explicitly found: the differential topological properties of fold maps affect the differentiable structures of the homotopy spheres.
Note also that every $7$-dimensional smooth homotopy sphere $M$ is represented as a connected sum of finitely many homotopy spheres represented as total spaces of linear bundles over $S^4$ whose fiber is diffeomorphic to $S^3$.
\subsection{Reeb spaces.}
For a continuous map $c:X \rightarrow Y$, we can define an equivalence relation ${\sim}_c$ on $X$ by the following rule: $p_1 {\sim}_c p_2$ if and only if $p_1$ and $p_2$ are in a same connected component of a preimage $c^{-1}(q)$ for some
$q \in Y$. We call the quotient space $W_c:=X/{{\sim}_c}$ the {\it Reeb space} of $c$. For a fold map satisfying the two conditions on the restriction to the singular set in subsection \ref{subsec:1.2}, the Reeb space is a $\dim Y$-dimensional polyhedron.
This property holds for smooth maps of wider classes such that the dimensions of the  manifolds of the domains are greater than those of the targets. See \cite{kobayashisaeki} and \cite{shiota} for example. The following example presents a fundamental and important fact in the next subsection.
\begin{Ex}
\label{ex:1}
For a special generic map, the Reeb space is a manifold we can smoothly immerse into the manifold of the target.
\end{Ex}
\subsection{Construction.}
\subsubsection{Surgery operations to manifolds and maps to construct new fold maps.}
\begin{Def}
\label{def:3}
For a fold map $f:M \rightarrow N$, let $P$ be a connected component of $(W_f-q_f(S(f))) \bigcap {\bar{f}}^{-1}(N-f(S(f)))$, regarded as a manifold diffeomorphic to $\bar{f}(P) \subset N$.

Let $l>0$ and $l^{\prime} \geq 0$ be integers. Assume that there exist families $\{S_j\}_{j=1}^{l}$ of finitely many standard spheres and $\{N(S_j)\}_{j=1}^{l}$ of total spaces of linear bundles over these spheres whose fibers are diffeomorphic to unit disks. $S_j$ also denotes the image of the section obtained by choosing the origin for each fiber diffeomorphic to the unit disk for each $N(S_j)$. Assume that the dimensions of $N(S_j)$ are always $n$
and that there exist smooth immersions $c_j:N(S_j) \rightarrow P$ satisfying the following properties.
\begin{enumerate}
\item $f {\mid}_{f^{-1}({\bigcup}_{j=1}^l c_j(N(S_j)))} f^{-1}({\bigcup}_{j=1}^l c_j(N(S_j))) \rightarrow {\bigcup}_{j=1}^l c_j(N(S_j))$ gives a trivial smooth bundle whose fiber is diffeomorphic to $S^{m-n}$.
\item The family $\{{c_j} {\mid}_{\partial N(S_j)}:\partial N(S_j) \rightarrow P\}_{j=1}^l$ is normal.
\item The family $\{{c_j} {\mid}_{S_j}:S_j \rightarrow P\}_{j=1}^l$ is normal and the number of crossings is finite.
\item Let the set of all crossings of the family of the immersions $\{{c_j} {\mid}_{S_j}:S_j \rightarrow P\}_{j=1}^l$ be denoted by $\{p_{j^{\prime}}\}_{j^{\prime}=1}^{l^{\prime}}$. For each $p_{j^{\prime}}$, there exist one or two integers $1 \leq a(j^{\prime}),b(j^{\prime}) \leq l$ and small standard closed disks $D_{2j^{\prime}-1} \subset  S_{a(j^{\prime})}$ and $D_{2j^{\prime}} \subset S_{b(j^{\prime})}$ satisfying the following four properties.
\begin{enumerate}
\item $\dim D_{2j^{\prime}-1}=\dim S_{a(j^{\prime})}$ and $\dim D_{2j^{\prime}}=\dim S_{b(j^{\prime})}$.
\item $p_{j^{\prime}}$ is in the images of the immersions $p_{j^{\prime}} \in c_{a(j^{\prime})}({\rm Int } D_{2j^{\prime}-1})$ and $p_{j^{\prime}} \in c_{b(j^{\prime})}({\rm Int } D_{2j^{\prime}})$.
\item If $a(j^{\prime})=b(j^{\prime})$, then $D_{2j^{\prime}-1}\bigcap D_{2j^{\prime}}$ is empty.
\item If we restrict the bundle $N(S_{a(j^{\prime})})$ over the sphere to $D_{2j^{\prime}-1}$ and the bundle $N(S_{b(j^{\prime})})$ over the sphere to $D_{2j^{\prime}}$, then the images of the total spaces of these resulting bundles by $c_{a(j^{\prime})}$ and $c_{b(j^{\prime})}$ agree as subsets in ${\mathbb{R}}^n$: the restrictions of the immersions to these spaces are embeddings.
\item The set of all crossings of the family $\{{c_j} {\mid}_{\partial N(S_j)}:\partial N(S_j) \rightarrow P\}_{j=1}^l$ is the disjoint union of the $l^{\prime}$ corners of the subsets just before each of which is for $1 \leq j^{\prime} \leq l^{\prime}$. 
\end{enumerate}
\end{enumerate}
In this situation, the family $\{(S_j,N(S_j),c_j:N(S_j) \rightarrow P)\}_{j=1}^{l}$ is said to be a {\it normal system of submanifolds} compatible with $f$. 
\end{Def}
In the situation of Definition \ref{def:3}, let $\{N^{\prime}(S_j) \subset N(S_j)\}_{j=1}^{l}$ be a family of total spaces of subbundles of $\{N(S_j)\}_{j=1}^{l}$ over the manifolds whose fibers are standard closed disks. We assume that the radii are all $0<r<1$. If we take $r$ suitably, then same properties as presented in Definition \ref{def:3} hold: we
can obtain another family $\{(S_j,N^{\prime}(S_j),{c_j} {\mid}_{N^{\prime}(S_j)}:N^{\prime}(S_j) \rightarrow P)\}_{j=1}^l$, regarded as a normal system of submanifolds compatible with $f$. We can identify each fiber, which is a standard closed disk of radius $r$ with a unit disk centered at the origin in the Euclidean space via the diffeomorphism mapping $t$ to $\frac{1}{r} t$.

\begin{Def}
\label{def:4}
The familiy $\{(S_j,N(S_j),c_j:N(S_j) \rightarrow P)\}_{j=1}^{l}$ is said to be a {\it wider normal system supporting} the normal system of submanifolds $\{(S_j,N^{\prime}(S_j),{c_j} {\mid}_{N^{\prime}(S_j)}:N^{\prime}(S_j) \rightarrow P)\}_{j=1}^{l}$ compatible with $f$.
\end{Def}

\begin{Def}
\label{def:5}
For a fold map $f:M \rightarrow N$ and an integer $l>0$, let $P$ be a connected
component of $(W_f-q_f(S(f))) \bigcap {\bar{f}}^{-1}(N-f(S(f)))$ and let $\{(S_j,N(S_j),c_j:N(S_j) \rightarrow P)\}_{j=1}^{l}$ be a normal system of submanifolds compatible with $f$.
Let $\{(S_j,N^{\prime}(S_j),{c_j}^{\prime}:N^{\prime}(S_j) \rightarrow P)\}_{j=1}^{l}$ be also a normal  system of submanifolds compatible with $f$. 
Assume also that $\{(S_j,N(S_j),c_j:N(S_j) \rightarrow P)\}_{j=1}^{l}$ is 
a wider normal system supporting $\{(S_j,N^{\prime}(S_j),{c_j}^{\prime}:N^{\prime}(S_j) \rightarrow P)\}_{j=1}^{l}$ .
Assume that we can construct a stable fold map $f^{\prime}$ on an $m$-dimensional closed manifold $M^{\prime}$ into ${\mathbb{R}}^n$ satisfying the following properties.
\begin{enumerate}
\item $Q:=f^{-1}({\bigcup}_{j=1}^l c_j(N(S_j)))$ is defined and $M-{\rm Int} Q$ is not empty. $M-{\rm Int} Q$ is regarded as an $m$-dimensional compact submanifold of $M^{\prime}$ via a suitable smooth embedding $e:M-{\rm Int} Q \rightarrow M^{\prime}$.
\item $f {\mid}_{M-{\rm Int} Q}={f }^{\prime} \circ e {\mid}_{M-{\rm Int} Q}$ holds.
\item ${f}^{\prime}(S({f}^{\prime}))$ is the disjoint union of $f(S(f))$ and ${\bigcup}_{j=1}^n c_j(\partial N^{\prime}(S_j))$.
\item The indices of points in the preimage of new connected components in the resulting singular value set are all $1$.
\item The preimage of each regular value $p$ sufficiently close to the union ${\bigcup}_{j=1}^l c_j(S_j)$ is a disjoint union of standard spheres.
\end{enumerate}
This enables us to define a procedure of constructing $f^{\prime}$ from $f$ and we call this an {\it ATSS} to $f$. We call the union ${\bigcup}_{j=1}^l c_j(S_j)$ the {\it generating image} of the operation.
\end{Def}
\begin{Prop}
\label{prop:1}
In the situation of Definition \ref{def:3} {\rm (}Definition \ref{def:5}{\rm )}, for the normal system of submanifolds compatible with $f$, we can do an ATSS to $f$ so that the ${\bigcup}_{j=1}^l c_j(S_j)$ is the generating image of the operation
as in Definition \ref{def:5}.
\end{Prop}
\begin{proof}
We present a local fold map around $p_{j^{\prime}}$ in Definition \ref{def:3}. This is also presented in \cite{kitazawa8} with FIGURE 2 and \cite{kitazawa9}. $S^{0}$ is the two point set with the discrete topology.
Set $k_1,k_2>0$ be integers.
We set ${D^{k_1}}_{\frac{1}{2}}$ as the set of all points $x \in {\mathbb{R}}^{k_1}$ satisfying $||x|| \leq \frac{1}{2}$.
We construct a trivial smooth bundle over ${D^{k_1}}_{\frac{1}{2}}$ whose fiber is diffeomorphic to $S^{k_2}$.
We also set a Morse function ${\tilde{f}}_{k_2,0}$ on a manifold obtained by removing the interior of a ($k_2+1$)-dimensional standard closed disk embedded smoothly in the interior of $S^{k_2} \times [-1,1]$ onto $[\frac{1}{2},\frac{3}{2}] \subset (0,+\infty) \subset \mathbb{R}$ satisfying the following four properties.
\begin{itemize}
\item The preimage of the minimum coincides with the disjoint union of two connected components of the boundary.
\item The preimage of the maximum coincides with one connected component of the boundary.
\item There exists exactly one singular point, and the singular point is in the interior of the manifold.
\item The value at the singular point is $1$.
\end{itemize}
We glue the projection of the trivial bundle over ${D^{k_1}}_{\frac{1}{2}}$ and the map ${\tilde{f}}_{k_2,0} \times {\rm id}_{\partial {D^{k_1}}_{\frac{1}{2}}}:[\frac{1}{2},+\infty) \times \partial {D^{k_1}}_{\frac{1}{2}}$. By gluing these maps suitably, we have a desired smooth map onto $\{x \in {\mathbb{R}}^{k_1} \mid ||x||  \leq \frac{3}{2}\}$.
See also \cite{kitazawa} and \cite{kitazawa4} for this map.
We can decompose the manifold of the domain into two compact manifolds with boundaries and the restriction to each manifold is as follows.
\begin{enumerate}
\item A surjection giving a trivial smooth bundle whose fiber is diffeomorphic to $D^{k_2}$.
\item A surjection such that the regular value set consist of two connected components. The preimage of a regular value in each connected component is diffeomorphic to the following manifolds.
\begin{enumerate}
\item $D^{k_2}$.
\item $D^{k_2} \sqcup S^{k_2}$.
\end{enumerate}
\end{enumerate}

Let the latter map be denoted by ${\tilde{f_{{\rm SD}}}}_{k_1,k_2}$.

We can take a sufficiently small standard closed disk ${D^{\prime}}_{2j^{\prime}-1} \supset D_{2j^{\prime}-1}$ of dimension $\dim D_{2j^{\prime}-1}$ so that the relation $S_{a(j^{\prime})} \supset {D^{\prime}}_{2j^{\prime}-1} \supset {\rm Int} {D^{\prime}}_{2j^{\prime}-1} \supset D_{2j^{\prime}-1}$ holds. Similarly we can take a sufficiently small standard closed disk ${D^{\prime}}_{2j^{\prime}} \supset D_{2j^{\prime}}$ of dimension $\dim D_{2j^{\prime}}$ so that the similar relation holds.
These disks are of course smoothly embedded in $S_{a(j^{\prime})}$ and $S_{b(j^{\prime})}$.
We can construct a product map of the composition of ${\tilde{f_{{\rm SD}}}}_{\dim {D^{\prime}}_{2j^{\prime}-1},m-n}$ with a diffeomorphism onto $c_{a(j^{\prime})}({D^{\prime}}_{2j^{\prime}-1})$ and the identity map ${\rm id}_{c_{b(j^{\prime})}({D^{\prime}}_{2j^{\prime}})}$. We can construct a product map of the composition of ${\tilde{f_{{\rm SD}}}}_{\dim {D^{\prime}}_{2j^{\prime}},m-n}$ with a diffeomorphism onto $c_{b(j^{\prime})}({D^{\prime}}_{2j^{\prime}})$ and the identity map ${\rm id}_{c_{a(j^{\prime})}({D^{\prime}}_{2j^{\prime}-1})}$. We glue these two maps suitably to obtain a local map onto $c_{a(j^{\prime})}({D^{\prime}}_{2j^{\prime}-1}) \times c_{b(j^{\prime})}({D^{\prime}}_{2j^{\prime}})$: the manifold of the target is same as those of the two local maps.

After constructing this local map around each $p_{j^{\prime}}$, we can construct for the remaining part easily. Around new singular values whose preimages consist of exactly one point, we construct maps represented as product maps of ${\tilde{f}}_{m-n,0}$ and identity maps on suitable manifolds for suitable coordinates. Around new regular values, we construct trivial smooth bundles.

This completes the proof.
\end{proof}  

\subsubsection{Homology and cohomology groups of manifolds admitting the fold maps obtained by the operations.}
For a homology class $c$ of a polyhedron $X$, it is {\it represented} by a subpolyhedron $Y$ embedded in $X$ if at least one of the following two hold.
\begin{enumerate}
\item Let $Y$ be a closed and orientable manifold. For a {\it fundamental class}, which is the uniquely defined class containing cycles obtained canonically from $Y$ with an orientation, the value of the homomorphism induced from the embedding into $X$ there is $c$.
\item More generally, let $Y$ be a quotient space of a closed and orientable manifold such that the quotient map is also a PL map. For a fundamental class, which is the uniquely defined class containing cycles obtained canonically from $Y$ with an orientation, the value of the homomorphism induced from the composition of the embedding with the quotient map into $X$ there is $c$.
\end{enumerate}

The following proposition is an extension of an important proposition in \cite{kitazawa7}. We can show this in an almost similar way. 
\begin{Prop}
\label{prop:2}
We consider a situation as explained in Definition \ref{def:3}.
Let $m>n \geq 1$ be integers. Let $n$ be even. Let $M$ be an $m$-dimensional closed and connected manifold. For the normal system of submanifolds compatible with $f$, assume that $\dim S_j=\frac{n}{2}$ for any $j$. We also assume the relations $0<n-\frac{n}{2}=\frac{n}{2}<m-n,n<m-n+\frac{n}{2}=m-\frac{n}{2}<m$ and $m-n \neq n$.

In this situation, we can do an ATSS to $f$ such that the ${\bigcup}_{j=1}^l c_j(S_j)$ is the generating image of the operation and have a new map $f^{\prime}:M^{\prime} \rightarrow {\mathbb{R}}^n$ satisfying the following three properties.

\begin{itemize}
\item $H_{i}(M^{\prime};\mathbb{Z})$ is isomorphic to $H_{i}(M;\mathbb{Z}) \oplus {\mathbb{Z}}^l$ for $i=\frac{n}{2},m-n,n,m-\frac{n}{2}$.
\item $H_{i}(M^{\prime};\mathbb{Z})$ is isomorphic to $H_{i}(M;\mathbb{Z})$ for $i \neq \frac{n}{2},m-n,n,m-\frac{n}{2}$.
\item The total Stiefel-Whitney class of $M^{\prime}$ is $1 \in H^{0}(M^{\prime};\mathbb{Z}/2\mathbb{Z})$ if that of $M$ is so.
\end{itemize}

Furthermore, for the resulting Reeb space, we can construct the map satisfying the following two properties.
\begin{itemize}
\item $H_{i}(W_{f^{\prime}};\mathbb{Z})$ is isomorphic to $H_{i}(W_f;\mathbb{Z}) \oplus {\mathbb{Z}}^l$ for $i=\frac{n}{2},m-n,n,m-\frac{n}{2}$. Under suitable identifications of $H_{i}(M^{\prime};\mathbb{Z})$ with $H_{i}(M;\mathbb{Z}) \oplus {\mathbb{Z}}^l$ and $H_{i}(W_{f^{\prime}};\mathbb{Z})$ with $H_{i}(W_f;\mathbb{Z}) \oplus {\mathbb{Z}}^l$, the homomorphism ${q_{f^{\prime}}}_{\ast}$ between the homology groups induced from the quotient map $q_{f^{\prime}}$ maps an element $(0,p) \in H_{i}(M;\mathbb{Z}) \oplus {\mathbb{Z}}^l$ to $(0,p) \in H_{i}(W_f;\mathbb{Z}) \oplus {\mathbb{Z}}^l$. Under the identifications, ${q_{f^{\prime}}}_{\ast}$ maps an element $(p,0) \in H_{i}(M;\mathbb{Z}) \oplus {\mathbb{Z}}^l$ to $({q_f}_{\ast}(p),0) \in H_{i}(W_f;\mathbb{Z}) \oplus {\mathbb{Z}}^l${\rm :} the homomorphism ${q_f}_{\ast}$ is induced from $q_f:M \rightarrow W_f$ in a canonical way.
\item $H_{i}(W_{f^{\prime}};\mathbb{Z})$ is isomorphic to $H_{i}(W_f;\mathbb{Z})$ for $i \neq \frac{n}{2},m-n,n,m-\frac{n}{2}$.
\end{itemize}
\end{Prop}

We present a short sketch of a proof explaining about homology and cohomology classes essential in the discussions later.
For a commutative group $G$ and an element $g$ which is of infinite order, which we cannot represent as $kg^{\prime}$ for an integer satisfying $|k|>1$ and $g^{\prime} \in G$ and which yields an internal direct sum decomposition $<g> \oplus G^{\prime}$ of $G \supset G^{\prime}$ where $<g>$ is the group generated by the set $\{g\}$, we can define the {\it dual} ${\rm D}(g)$ of $g$ as a homomorphism into $\mathbb{Z}$ satisfying ${\rm D}(g)(g)=1$ and for any $G^{\prime}$ satisfying the same condition ${\rm D}(g)(G^{\prime})=\{0\}$. 
Moreover, if the group is a homology group, then we can regard the dual as a cohomology class in a canonical way and we regard the dual as the cohomology class obtained in this way. 


\begin{proof}[A short sketch of a proof]
The proof for the case where the immersion obtained as the restriction to the singular set is an embedding is already done in \cite{kitazawa7}. The construction and existence of self-intersections of the immersion do not influence on the original argument considerably.

In Definition \ref{def:5}, consider a fiber $D_{{\rm F},j}$ of the bundle $N(S_j)$ over $S_j$ such that $c_j(D_{{\rm F},j})$ contains no crossings of the family $\{c_j {\mid}_{\partial N(S_j)}\}$ of the immersions,
set $Q_{{\rm F},j}:=f^{-1}(c_j(D_{{\rm F},j}))$ and consider the restriction of $f$ to $Q_{{\rm F},j}$. 
Note that $Q_{{\rm F},j}$ is a compact manifold diffeomorphic to a manifold obtained in the following way. See also \cite{kitazawa} and \cite{kitazawa4} for this.
\begin{enumerate}
\item $E_0:=S^n \times S^{m-n}$.
\item Choose an equator $E_S$ of $S^n$, diffeomorphic to $S^{n-1}$.
\item We consider a small closed tubular neighborhood of $E_S \times \{{\ast}_{E_S}\} \subset E_0$ and remove its interior where ${{\ast}_{E_S}}$ is a suitable point.
\item The resulting compact manifold $E_D$ is the desired manifold. We can construct and need to construct two relations $E_S \times \{{{\ast}_{E_S}}^{\prime}\} \subset E_D \subset E_0$ and $\{{{\ast}_{E_S}}^{\prime \prime}\} \times S^{m-n} \subset E_D \subset E_0$ hold where ${{\ast}_{E_S}}^{\prime}$ and ${{\ast}_{E_S}}^{\prime \prime}$ are suitable points. We also need to construct that $\{{{\ast}_{E_S}}^{\prime \prime}\} \times S^{m-n}$ is a connected component of the preimage of some regular value. 
\end{enumerate}

We can construct a desired map $f^{\prime}:M^{\prime} \rightarrow {\mathbb{R}}^n$ on a suitable manifold $M^{\prime}$. We explain a generator of the $j$-th summand of ${\mathbb{Z}}^l$ in $H_{i}(M;\mathbb{Z}) \oplus {\mathbb{Z}}^l$, isomorphic to $H_{i}(M^{\prime};\mathbb{Z})$, for $i=\frac{n}{2},n,m-n,m-\frac{n}{2}$.

For $i=\frac{n}{2}$, it is represented by $E_S \times \{{{\ast}_{E_S}}^{\prime}\}$ before for each $j$. For $i=n$, it is represented by the product of the image of an embedding of $S_j$ into $f^{-1}(c_j(S_j))$ such that the composition with $f$ and the immersion $c_j$ agree and $E_S \times \{{{\ast}_{E_S}}^{\prime}\}$ before for each $j$: we can take an embedding of $S_j$ by the structure of the manifold and the map. For $i=m-n$, it is represented by a connected component of the preimage of a regular value, diffeomorphic to $S^{m-n}$, before, for each $j$. For $i=m-\frac{n}{2}$, it is represented by the product of the image of an embedding of $S_j$ into $f^{-1}(c(S_j))$ such that the composition with $f$ and the immersion $c_j$ agree and a connected component of a preimage, diffeomorphic to $S^{m-n}$, before for each $j$: we can take an embedding of $S_j$ by the structure of the manifold and the map.

We can see that the remaining properties hold by the way of this construction.
\end{proof}

\begin{Rem}
\label{rem:1}
On Proposition \ref{prop:2}, originally, we consider a situation as explained in Definition \ref{def:3} as the following.
Let $m>n \geq 1$ be integers. Let $M$ be an $m$-dimensional closed and connected manifold. The normal system of submanifolds compatible with $f$ satisfies the following three.
\begin{enumerate}
\item The immersions of standard spheres are embeddings and that the family of the immersions has no crossings.
\item $\dim S_j=k \leq \frac{n}{2}$ for any $j$.
\item The relations $0<k<m-n,n<m-n+k<m$ and $m-n \neq n$ hold.
\end{enumerate}
In this situation, we have an essentially similar result. This will be used in the proof of the main theorem or Theorem \ref{thm:5}.
\end{Rem}
\subsubsection{The main theorem.}
The following proposition plays essential roles in the proof of the main theorem.
\begin{Prop}
\label{prop:3}
We consider a situation as explained in Proposition \ref{prop:2}.
Let $m>n \geq 1$ be integers. Assume that $n$ is divisible by $4$. 
For the normal system of submanifolds compatible with $f$, assume that $\dim S_j=\frac{n}{2}$ for any $j$ as in Proposition \ref{prop:2}. We also assume the relation on the dimensions $0<n-\frac{n}{2}<m-n<n<m-n+\frac{n}{2}<m$ as in the proposition. Furthermore, let $H:=(h_{j_1,j_2})$ be an $l \times l$ symmetric matrix values of whose components are integers satisfying $h_{j,j}=0$ for $1 \leq j \leq l$ where $h_{j_1,j_2}$ denotes the $(j_1,j_2)$-th element.  
We also assume the following two.
\begin{enumerate}
\item For any pair $(c_{j_1} {\mid}_{S_{j_1}},c_{j_2} {\mid}_{S_{j_2}})$ of the immersions, there exist sufficiently many crossings. 
\item Identifications of homology groups used in Proposition \ref{prop:2} are given.
\end{enumerate}
For $H_{\frac{n}{2}}(M;\mathbb{Z}) \oplus {\mathbb{Z}}^l$, let $e_j$ be the element represented as $(0,{e_{0,j}}) \in H_{\frac{n}{2}}(M;\mathbb{Z}) \oplus {\mathbb{Z}}^l${\rm :} $e_{0,j}$ is the sequence of $l$ integers the value of whose $j$-th component is $1$ and those of whose remaining components are all $0$.
For $H_{n}(M;\mathbb{Z}) \oplus {\mathbb{Z}}^l$, let ${e_j}^{\prime}$ be the element represented as $(0,{e_{0,j}}^{\prime}) \in H_{n}(M;\mathbb{Z}) \oplus {\mathbb{Z}}^l$: ${e_{0,j}}^{\prime}$ is the sequence of $l$ integers whose $j$-th component is $1$ and the other components are $0$.

In this situation, we can do an ATSS to $f$ such that the ${\bigcup}_{j=1}^l c_j(S_j)$ is the generating image of the operation and have a new map $f^{\prime}:M^{\prime} \rightarrow {\mathbb{R}}^n$ satisfying the same properties as ones in Proposition \ref{prop:2} and the following additional properties where ${q_{f^{\prime}}}^{\ast}:H^i(W_{f^{\prime}};\mathbb{Z}) \rightarrow H^i(M^{\prime};\mathbb{Z})$ is the homomorphism induced from the quotient map
 $q_{f^{\prime}}:M^{\prime} \rightarrow W_{f^{\prime}}$ for each $i$. Note that for $e_j$ and ${e_j}^{\prime}$, we can define duals ${\rm D}(e_j)$ and ${\rm D}({e_j}^{\prime})$, respectively. Note also that by virtue of Proposition \ref{prop:2}, we can define the duals ${\rm D}({q_{f^{\prime}}}_{\ast}(e_j))$ and ${\rm D}({q_{f^{\prime}}}_{\ast}({e_j}^{\prime}))$.
\begin{itemize}
\item Two relations ${q_{f^{\prime}}}^{\ast}({\rm D}({q_{f^{\prime}}}_{\ast}(e_j)))={\rm D}(e_j)$ and ${q_{f^{\prime}}}^{\ast}({\rm D}({q_{f^{\prime}}}_{\ast}({e_j}^{\prime})))={\rm D}({e_j}^{\prime})$ hold.
\item The product of ${\rm D}(e_{j_1})$ and ${\rm D}(e_{j_2})$ is $h_{j_1,j_2}{\rm D}({e_{j_1}}^{\prime})+h_{j_2,j_1}{\rm D}({e_{j_2}}^{\prime})$. 
\end{itemize}
\end{Prop}
\begin{proof}
Main ingredients of the proof are essentially in the proofs of important propositions including main propositions in \cite{kitazawa8} and \cite{kitazawa9}. We will concentrate on expositions of several essential parts.

We can see the properties in Proposition \ref{prop:2} and the first property additionally presented here by the way of the construction together with the observation of the topologies of the original Reeb space $W_f$ and the resulting one $W_{f^{\prime}}$, we will explain here. 

$W_{f^{\prime}}$ is simple homotopy equivalent to the polyhedron obtained by attaching a polyhedron to a subpolyhedron ${\bar{f}}^{-1}({\bigcup}_{j=1}^l c_j(S_j)) \subset W_f$. The dimension of the subpolyhedron is $\frac{n}{2}$.

We explain about the polyhedron. We prepare $l$ copies of a manifold diffeomorphic to $S^{\frac{n}{2}} \times S^{\frac{n}{2}}$. Each copy is denoted by $S_{B,j} \times S_{F,j}$ for $1 \leq j \leq l$ where $S_{B,j}$ and $S_{F,j}$ are diffeomorphic to $S^{\frac{n}{2}}$. 

We choose pairs $((D_{B,2j^{\prime}-1} \times D_{F,2j^{\prime}-1} \subset S_{B,p(2j^{\prime}-1)} \times S_{F,p(2j^{\prime}-1)}),(D_{B,2j^{\prime}} \times D_{F,2j^{\prime}} \subset S_{B,p(2j^{\prime})} \times S_{F,p(2j^{\prime})}))$ of products of standard closed disks whose dimensions are $\frac{n}{2}$ where $p$ is a function mapping the integer to an integer $1 \leq i \leq l$: $j^{\prime}$ is an arbitrary positive integer smaller than a sufficiently large integer. We also need to choose the products of disks disjointly. We identify the two products of disks for each pair. We can do an ATSS so that the following properties hold.
\begin{itemize}
\item The resulting Reeb space is simple homotopy equivalent to one obtained by gluing $S_{B,j} \times S_{F,j}$, represented via notation respecting the original manifolds diffeomorphic to $S^{\frac{n}{2}} \times S^{\frac{n}{2}}$, by attaching ${\bigcup}_{j} S_{B,j} \times \{{\ast}_{j}\} (\subset S_{B,j} \times S_{F,j})$, to the subpolyhedron ${\bar{f}}^{-1}({\bigcup}_{j=1}^l c_j(S_j)) \subset W_f$.  
\item ${q_{f^{\prime}}}_{\ast}(e_j)$ is represented by $\{{{\ast}^{\prime}}_{j}\} \times S_{F,j} \subset S_{B,j} \times S_{F,j}$, after the manifold is attached to $W_f$, in the resulting Reeb space $W_{f^{\prime}}$.
\item $S_{B,j} \times \{{\ast}_j\}$ is, after attached to $W_f$, in the resulting Reeb space $W_{f^{\prime}}$, a subpolyhedron by which the class $({q_f}_{\ast}(c_{M,j}),{\Sigma}_{j^{\prime \prime}=1}^{l} h_{j,j^{\prime \prime}} {q_{f^{\prime}}}_{\ast}(e_{j^{\prime \prime}})) \in H_{\frac{n}{2}}(W_f;\mathbb{Z}) \oplus {\mathbb{Z}}^l$ is represented where $c_{M,j}$ is a suitable class.
\end{itemize}
The main ingredients are the assumption that for any pair $(c_{j_1} {\mid}_{S_{j_1}},c_{j_2} {\mid}_{S_{j_2}})$ of the immersions of standard spheres, there exist sufficiently many crossings and the arguments on several homology classes in the proofs of Proposition 3 of \cite{kitazawa8} and Proposition 2 of \cite{kitazawa9} with Proposition \ref{prop:1} of the present paper.

This yields the fact that the second and third additional properties hold. More precise expositions on this with the argument on the explicit ATSS just before will be also seen in the proofs of important propositions and theorems in \cite{kitazawa8} and \cite{kitazawa9}.
\end{proof}

The following special generic maps also play important roles in the construction.

\begin{Ex}
\label{ex:2}
Let $m>n \geq 2$ and $l>0$ be integers. An $m$-dimensional closed and connected manifold $M$ represented as a connected sum of $l$ manifolds each of which is diffeomorphic to $S^{l_j} \times S^{m-l_j}$ for each integer $1 \leq l_j \leq n-1$ admits a special generic map $f:M \rightarrow {\mathbb{R}}^n$ such that the following properties hold.
\begin{enumerate}
\item $W_f$ is represented as a boundary connected sum of $l$ manifolds each of which is diffeomorphic to $S^{l_j} \times D^{n-l_j}$ for each integer $1 \leq l_j \leq n-1$.
\item $f {\mid}_{S(f)}$ is an embedding.
\item We can obtain a trivial linear bundle over $\partial W_f$ whose fiber is diffeomorphic to $D^{m-n+1}$ by restricting $f$ to the preimage of a small collar neighborhood and considering the composition of this restriction with the canonical projection to
$\partial W_f$.
\item We can obtain a trivial smooth bundle over the complement of the interior of the collar neighborhood before in $W_f$ whose fiber is diffeomorphic to $S^{m-n}$ by considering the restriction of $f$ to the preimage of this interior of the collar neighborhood.
\item The homomorphism ${q_f}_{\ast}$ between the homology groups maps the class represented by $S^{l_j} \times \{{\ast}_1\} \subset S^{l_j} \times S^{m-l_j}$ in the connected sum to the
class represented by $S^{l_j} \times \{{\ast}_2\} \subset S^{l_j} \times {\rm Int} D^{n-l_j} \subset S^{l_j} \times D^{n-l_j}$ in the boundary connected sum $W_f$.
\end{enumerate}
\end{Ex}

The following theorem is the main theorem.

\begin{Thm}
\label{thm:5}
Let $A$, $B$ and $B^{\prime}$ be free finitely generated commutative groups of ranks $a$, $b$ and $b^{\prime}$, respectively. Let $\{a_j\}_{j=1}^{a}$, $\{b_j\}_{j=1}^{b}$ and $\{{b_j}^{\prime}\}_{j=1}^{b^{\prime}}$ be bases of $A$, $B$ and $B^{\prime}$, respectively.
For each integer $1 \leq i \leq b$, let $\{a_{i,j}\}_{j=1}^{a}$ be a sequence of integers. Let $H:=(h_{j_1,j_2})$ be a $b \times b$ symmetric matrix whose components are integers satisfying $h_{j,j}=0$ for $1 \leq j \leq b$ where $h_{j_1,j_2}$ denotes the values of the
$(j_1,j_2)$-th component.

In this situation, there exist a $7$-dimensional closed and simply-connected spin manifold $M$ and a fold map $f:M \rightarrow {\mathbb{R}^4}$ satisfying the following properties.
\begin{enumerate}
\item $H_2(M;\mathbb{Z})$ is isomorphic to $A \oplus B$ and $H_4(M;\mathbb{Z})$ is isomorphic to $B \oplus B^{\prime}$. $H_3(M;\mathbb{Z})$ is free.
\item $H^j(M;\mathbb{Z})$ is isomorphic to $H_j(M;\mathbb{Z})$. Via suitable isomorphisms ${\phi}_2:A \oplus B \rightarrow H_2(M;\mathbb{Z})$ and ${\phi}_4:B \oplus B^{\prime} \rightarrow H_4(M;\mathbb{Z})$, we can identify $H_2(M;\mathbb{Z})$ with
$A \oplus B$ and $H_4(M;\mathbb{Z})$ with $B \oplus B^{\prime}$. Furthermore, we can define the duals ${a_j}^{\ast}$, ${b_{j,2}}^{\ast}$, ${b_{j,4}}^{\ast}$ and ${{b_j}^{\prime}}^{\ast}$ of ${\phi}_2((a_j,0))$, ${\phi}_2((0,b_j))$, ${\phi}_4((b_j,0))$ and ${\phi}_4((0,{b_j}^{\prime}))$, respectively, and we can canonically obtain bases of $H^2(M;\mathbb{Z})$ and $H^4(M;\mathbb{Z})$. Let the composition of ${\phi}_4$ with an isomorphism mapping the elements of the basis to their Poincar\'e duals
be denoted by ${\phi}_{4,3}:B \oplus B^{\prime} \rightarrow H^3(M;\mathbb{Z})$. Let the composition of ${\phi}_2$ with an isomorphism mapping the elements of the basis to their Poincar\'e duals be denoted by
${\phi}_{2,5}:A \oplus B \rightarrow H^5(M;\mathbb{Z})$.

For the products in $H^i(M;\mathbb{Z})$ for $i=4,5,7$, properties presented in Theorem \ref{thm:1} {\rm (}\ref{thm:1.2}{\rm )} hold except the following cases. Note also that "$C$" in Theorem \ref{thm:1} is replaced by $B^{\prime}$ and "$c_j$" in Theorem \ref{thm:1} is replaced by ${b_j}^{\prime}$ for example.
\begin{enumerate}
\item The first list of {\rm (}\ref{thm:1.2.1}{\rm )} must be replaced. Instead, the product of ${a_{j_1}}^{\ast}$ and ${a_{j_2}}^{\ast}$ vanishes and that of ${b_{j_1,2}}^{\ast}$ and ${b_{j_2,2}}^{\ast}$ is $h_{j_1,j_2} {b_{j_1,4}}^{\ast}+h_{j_2,j_1} {b_{j_2,4}}^{\ast}$ for any pair $(j_1,j_2) \in {{\mathbb{N}}_{\leq a}}^2$ and $(j_1,j_2) \in {{\mathbb{N}}_{\leq b}}^2$ respectively.
\item The fifth list of {\rm (}\ref{thm:1.2.2}{\rm )} must be replaced. Instead, the product of ${b_{i,2}}^{\ast}$ and ${\phi}_{4,3}((b_{i},0))$ is ${\Sigma}_{j=1}^{a} a_{i,j}{\phi}_{2,5}((a_j,0))+{\Sigma}_{j=1}^{b} h_{i,j}{\phi}_{2,5}((0,b_j)) \in H^5(M;\mathbb{Z})$ for any number $i \in {\mathbb{N}}_{\leq b}$.
\end{enumerate}
\item The 3rd and the 5th Stiefel-Whitney classes of $M$ vanish.
\item Let $d$ be the isomorphism $d:H_4(M;\mathbb{Z}) \rightarrow H^4(M;\mathbb{Z})$ defined by corresponding the duals. The first Pontryagin class of $M$ is regarded as $4d \circ {\phi}_4(p) \in H^4(M;\mathbb{Z})$. The 4th Stiefel-Whitney class of $M$ vanishes.
\item The immersion $f {\mid}_{S(f)}$ is normal.
\item The index of each singular point is $0$ or $1$.
\item Preimages of regular values in the image of $f$ are diffeomorphic to $S^3$, $S^3 \sqcup S^3$ or $S^3 \sqcup S^3 \sqcup S^3$.
\end{enumerate}
\end{Thm}
Note that the statement is same as that of Theorem \ref{thm:1} except several properties. The proof is also similar to the original theorem. New ingredients are discussions using Proposition \ref{prop:3} for example.
See also \cite{kitazawa6} and \cite{kitazawa7} to understand the proof more rigorously, for example.
\begin{proof}
We consider a special generic map $f_0$ presented in Example \ref{ex:2} into ${\mathbb{R}}^4$ on a manifold represented as a connected sum of $a \geq 0$ copies of $S^2 \times S^5$.
We identify $a_j \in A$ in the assumption with a generator of the 2nd homology group of the $j$-th copy of $S^2 \times S^5$ and a suitable class in the Reeb space $W_{f_0}$ presented in the example.

We can choose a family $\{c_j:S_j \rightarrow {\rm Int} W_{f_0}\}_{j=1}^b$ of immersions of $2$-dimensional standard spheres into the interior of the image of $f_0$ or $W_{f_0}$ so that the family is normal and that for any pair of the immersions, there exist sufficiently many crossings as assumed in Proposition \ref{prop:3}.

We can also choose $b^{\prime}$ points embedded disjointly in the complement of the union of the images of the immersions before in the interior of the image of $f_0$ or $W_{f_0}$.
We can take them so that the following properties hold.
\begin{itemize}
\item Let a generator of the 2nd homology group of the $j$-th copy of $D^2 \times S^2$ be denoted by $a_j$ in the image, represented as a boundary connected sum of $a$ copies of $D^2 \times S^2$. It can be identified with the element $a_j \in A$ of the assumption.
\item Let a fundamental class of $S_i$ be denoted by $[S_i]$. ${c_i}_{\ast}([S_i])={\Sigma}_{j=1}^{a} a_{i,j} a_j \in H_2(W_{f_0};\mathbb{Z})$: ${c_i}_{\ast}$ is the homomorphism between the homology groups induced from $c_i$.
\end{itemize}

We can do an ATSS so that the resulting map and the $7$-dimensional manifold satisfy the fifth, sixth and seventh properties.

The first property and statements on 2nd and 4-th homology groups of the second property hold for a suitable ATSS. We use Proposition \ref{prop:2} setting $(m,n)=(7,4)$ in Proposition \ref{prop:2} and Remark \ref{rem:1} setting $(m,n,k)=(7,4,0)$.

We can identify the homology class represented by $E_S \times \{{{\ast}_{E_S}}^{\prime}\}$ in (the proof of) Proposition \ref{prop:2} for each $S_i$. The class represented by this is identified with $b_i \in B$ in the assumption as we have done for $a_j$ and we thus obtain a desired isomorphism ${\phi}_2:A \oplus B \rightarrow H_2(M;\mathbb{Z})$.
We can apply Proposition \ref{prop:2} for the $4$-th homology group. In addition, for the $b^{\prime}$ points $p_i$, we apply Remark \ref{rem:1}. We can consider similar identifications for ${b_i}^{\prime}$. This yields a desired isomorphism ${\phi}_4:B \oplus B^{\prime} \rightarrow H_4(M;\mathbb{Z})$.

We prove properties on products of cohomology classes in the second property of the seven properties for a suitable ATSS.

We prove the properties for cases where products are in $H^4(M;\mathbb{Z})$.

The product of ${a_{i_1}}^{\ast}$ and ${a_{i_2}}^{\ast}$ vanish by the identifications before.
Consider the product of a cocycle representing the dual ${b_{i,2}}^{\ast}$ of the class $b_{i}$ and a cocycle representing the dual ${a_{j}}^{\ast}$ of the class $a_{j}$ and value at the cycle represented as the tensor product of a cycle in the class ${\Sigma}_{j=1}^{a} a_{i,j}a_j$ and a cycle in the class represented by $E_S \times \{{{\ast}_{E_S}}^{\prime}\}$ in the proof of Proposition \ref{prop:2} (we consider a cohomology class in the product $M \times M$ defined in a canonical way and the product is obtained as the pull-back via the diagonal map from $M$ into $M \times M$). For the product of ${b_{i_1,2}}^{\ast}$ and ${b_{i_2,2}}^{\ast}$, apply Proposition \ref{prop:3} and this is a new ingredient in the proof.

We prove the properties for cases where products are in $H^5(M;\mathbb{Z})$.

We investigate several homomorphisms. ${\phi}_{4,3} \circ {{\phi}_{4}}^{-1}$ maps the 4-th homology class ${\phi}_4((b_j,0))$ or ${\phi}_4((0,{b_j}^{\prime}))$ before to the dual of the homology class represented by a connected component
of the preimage of a suitable regular value. We explain about ${\phi}_{2,5} \circ {{\phi}_2}^{-1}$. This maps each ($n-\dim S_i$)-th homology class represented by $E_S \times \{{{\ast}_{E_S}}^{\prime}\}$ in the proof of Proposition \ref{prop:2} to the dual of the homology class represented by a product of $\{{{\ast}_{E_S}}^{\prime \prime}\} \times S^{m-n} \subset E_D \subset E_0$ in the proof of Proposition \ref{prop:2} and the image of an embedding of $S_i$ into $f^{-1}(c_i(S_i))$ such that the composition with $f$ and the immersion $c_i$ agree in the proof of the proposition.

The first three products in Theorem \ref{thm:1} (\ref{thm:1.2.2}) vanish. In fact, for the pair of the cohomology classes, we can take corresponding cycles, representing the classes whose duals are these cohomology classes satisfying the following: each cohomology class does not vanish at the corresponding cycle and the cycles do not intersect.

We prove the fourth product in Theorem \ref{thm:1} (\ref{thm:1.2.2}).
We investigate several classes. The class ${c_{j_2}}_{\ast}([S_{j_2}])$ coincides with ${\Sigma}_{j_1=1}^{a} a_{j_2,j_1}a_{j_1} \in H_2(W_{f_0};\mathbb{Z})$. ${\phi}_{2,5}$ maps each ($n-\dim S_{j_2}$)-th homology class represented by $E_S \times \{{{\ast}_{E_S}}^{\prime}\}$ in the proof of Proposition \ref{prop:2} to the dual of the homology class represented by a product of the image of an embedding of $S_{j_2}$ into $f^{-1}(c_{j_2}(S_{j_2}))$ such that the composition with $f$ and the immersion $c_{j_2}$ agree in the proof of the proposition and $\{{{\ast}_{E_S}}^{\prime \prime}\} \times S^{m-n} \subset E_D \subset E_0$ there. We can see that the product of ${a_{j_1}}^{\ast}$ and ${\phi}_{4,3}((b_{j_2},0))$ is $a_{j_2,j_1}{\phi}_{2,5}((0,b_{j_2})) \in H^5(M;\mathbb{Z})$ for any $j_1$ and $j_2$.

We prove the fifth case. The proof is a new ingredient. ${\Sigma}_{j=1}^{a} a_{i,j}a_j \in H_2(W_{f_0};\mathbb{Z})$ and ${c_{i}}_{\ast}([S_{i}])$ agree. Thus, by virtue of the discussion in the proof of Proposition \ref{prop:2} and by considering the class to which the isomorphism ${\phi}_{2,5} \circ {{\phi}_2}^{-1}$ maps $a_j \in H_2(M;\mathbb{Z})$, for the product of ${b_{i,2}}^{\ast}$ and ${\phi}_{4,3}((b_{i},0))$, the value at the class whose dual is ${\phi}_{2,5} \circ {{\phi}_2}^{-1}(a_j) \in H^5(M;\mathbb{Z})$ is $a_{i,j}$. Here the identification of elements in $A$ and the homology group is also considered. The ATSS we do here respects Proposition \ref{prop:3} and this requires a new ingredient. Consider the value at the class whose dual is ${\phi}_{2,5} \circ {{\phi}_2}^{-1}(b_j) \in H^5(M;\mathbb{Z})$ where the identification as before is considered. The value is $h_{i,j}$.

This completes the proof of the last case and the proof for cases where products are in $H^5(M;\mathbb{Z})$.

We explain about cases where products are in $H^7(M;\mathbb{Z})$. Proofs of the facts that the products vanish are done in the following way: for the pair of the cohomology classes, we can take corresponding cycles disjointly, representing the classes whose duals are these cohomology classes so that each class does not vanish at these cycles. Remaining cases can be shown by virtue of the definitions of ${\phi}_{2,5}$ and ${\phi}_{4,3}$ and Poincare duality theorem.

For the third property, Proposition \ref{prop:2} and Remark \ref{rem:1} complete the proof.

We show the fourth property. $H_3(M;\mathbb{Z})$ is generated by a finite set consisting of the classes represented by connected components of preimages of regular values. We exchange small tubular
neighborhoods of the connected components, yielding trivial linear bundles over $D^4$ whose fiber is diffeomorphic to $S^3$, and change the map. More precisely, we change bundle isomorphisms for identifications between the boundaries: each connected component of the boundary is $S^3 \times \partial D^4$. Classical important facts on Pontryagin classes of linear bundles over $S^4$ whose fibers are diffeomorphic to $S^3$ yield the fourth property. See also \cite{milnorstasheff} and \cite{steenrod} for example. More precise expositions on this are also in \cite{kitazawa7} for example.

This completes the proof.
\end{proof}

It is important that this theorem explains the cohomology ring of the manifold $M$ completely as Theorem \ref{thm:1} does by virtue of fundamental theorems from algebraic topology such as Poincar\'e duality theorem and universal coefficient theorem.

We can show the following proposition as an extension of Theorem \ref{thm:3}. We omit a proof and proofs are left to readers: see the original article \cite{kitazawa7}.

\begin{Prop}
In the situation of Theorem \ref{thm:5}, if at least one of the following conditions is satisfied, then $M$ admits no special generic maps into ${\mathbb{R}}^4$.
\begin{enumerate}
\item $p \in B \oplus C$ is not zero.
\item In some sequence $\{a_{i,j}\}_{j=1}^{a}$, at least one non-zero number exists.
\item There exists at least one component of $H$ whose value is is not zero.
\end{enumerate}
\end{Prop}

\section{An additional remark on integral cohomology rings of $7$-dimensional closed and simply-connected manifolds admitting these fold maps.}
\label{sec:3}
We add a remark on the class of $7$-dimensional closed, simply-connected and spin manifolds in Theorems \ref{thm:1} and \ref{thm:5}. 

Consider the cases where $a+b=3$ and $b^{\prime}=c=0$ in these theorems. As presented in the introduction, \cite{kitazawa8} presents a study of topological properties of Reeb spaces of fold maps having topological properties similar to ones in the present paper and Example 2 in this preprint presents the case where $a=1$, $b=2$, $a_{1,1} \neq 0$, $a_{1,2} \neq 0$, and $h_{1,2}=h_{2,1} \neq 0$ hold. In Theorem \ref{thm:1}, we can find a submodule of rank $2$ of $H^2(M;\mathbb{Z})$ consisting of elements whose squares vanish: more generally this holds in the case $a+b>2$ holds. The example implies that in this case for Theorem \ref{thm:5} we cannot find a submodule of rank $2$ of $H^2(M;\mathbb{Z})$ consisting of elements whose squares vanish. 

For example, let $a_{1,1}=1$ and $a_{1,2}=1$, and let $h_{1,2}=h_{2,1}=1$.
Consider the square of $r_1{a_1}^{\ast}+r_2{b_{1,2}}^{\ast}+r_3{b_{2,2}}^{\ast}$ for $r_1,r_2,r_3 \in \mathbb{Z}$. The square is $2r_1r_2{b_{1,4}}^{\ast}+r_2r_3({b_{1,4}}^{\ast}+{b_{2,4}}^{\ast})+2r_1r_3{b_{2,4}}^{\ast}$.
It vanishes if and only if at least one of the following two conditions.
\begin{enumerate}
\item $r_2=0$, either $r_1$ or $r_3$ is $0$.
\item $2r_1+r_3=0$, $r_2r_3-{r_3}^2=0$.
\end{enumerate}
In each cases, $(r_1,r_2,r_3)$ lies in a suitable real line in ${\mathbb{R}}^3$ containing the origin.

\begin{Cor}
The class of $7$-dimensional closed, simply-connected and spin manifolds $M$ in Theorem \ref{thm:5} is wider than that in Theorem \ref{thm:1}.
\end{Cor}
\section{Acknowledgement.}
\label{sec:4}
\thanks{The author is a member of JSPS KAKENHI Grant Number JP17H06128 "Innovative research of geometric topology and singularities of differentiable mappings"
\begin{center}
(https://kaken.nii.ac.jp/en/grant/KAKENHI-PROJECT-17H06128/: Principal investigator is Osamu Saeki).
\end{center}
This work was also supported by
\begin{center}
''The Sasakawa Scientific Research Grant" (2020-2002 : https://www.jss.or.jp/ikusei/sasakawa/).
\end{center}
}
We declare that all data supporting our present study are all contained in the present paper.
\section{Appendices.}
\label{sec:6}
We extend some notions, propositions and theorems in the present paper.

\begin{Def}
\label{def:6}
Let $l>0$ and $l^{\prime} \geq 0$ be integers. Assume that there exist families $\{S_j\}_{j=1}^{l}$ of compact and connected spaces each of which is regarded as a polyhedron obtained in the following.
\begin{itemize}
\item Prepare finitely many standard spheres.
\item Choose finitely many pairs of points in the spheres such that distinct pairs do not have common points.
\item Identify two points in each pair. Let $S(S_{j})$ denotes the set of all points in $S_j$ originally in some pair of the previous pairs.
\end{itemize}
Assume also that we can take a family $\{N(S_j)\}_{j=1}^{l}$ of compact and smooth manifolds each of which is regarded as a regular neighborhood of $S_j$. Assume that the dimensions of $N(S_j)$ are always $n$
and that there exist smooth immersions $c_j:N(S_j) \rightarrow P$ satisfying the following properties.
\begin{enumerate}
\item $f {\mid}_{f^{-1}({\bigcup}_{j=1}^l c_j(N(S_j)))}: f^{-1}({\bigcup}_{j=1}^l c_j(N(S_j))) \rightarrow {\bigcup}_{j=1}^l c_j(N(S_j))$ gives a trivial smooth bundle whose fiber is diffeomorphic to $S^{m-n}$.
\item The family $\{{c_j} {\mid}_{\partial N(S_j)}:\partial N(S_j) \rightarrow P\}_{j=1}^l$ is normal.
\item The family $\{{c_j} {\mid}_{S_j-S(S_j)}:S_j-S(S_j) \rightarrow P\}_{j=1}^l$ is normal and the number of crossings is finite.
\item The disjoint union ${\sqcup}_{j=1}^l {c_j} {\mid}_{{\sqcup}_{j=1}^l S(S_j)}:{\sqcup}_{j=1}^l S(S_j) \rightarrow P$ is injective. The image of the disjoint union ${\sqcup}_{j=1}^l {c_j} {\mid}_{{\sqcup}_{j=1}^l S_j-S(S_j)}:{\sqcup}_{j=1}^l S_j-S(S_j)\rightarrow P$ and the image of the map are disjoint.
\item Let the set of all crossings of the family of the immersions $\{{c_j} {\mid}_{S_j-S(S_j)}:S_j-S(S_j) \rightarrow P\}_{j=1}^l$ denoted by $\{p_{j^{\prime}}\}_{j^{\prime}=1}^{l^{\prime}}$. For each $p_{j^{\prime}}$, there exist one or two integers $1 \leq a(j^{\prime}),b(j^{\prime}) \leq l$ and small standard closed disks $D_{2j^{\prime}-1} \subset S_{a(j^{\prime})}$ and $D_{2j^{\prime}} \subset S_{b(j^{\prime})}$ satisfying the following four properties.
\begin{enumerate}
\item $\dim D_{2j^{\prime}-1}=\dim S_{a(j^{\prime})}$ and $\dim D_{2j^{\prime}}=\dim S_{b(j^{\prime})}$.
\item $p_{j^{\prime}}$ is in the images of the immersions $p_{j^{\prime}} \in c_{a(j^{\prime})}({\rm Int } D_{2j^{\prime}-1})$ and $p_{j^{\prime}} \in c_{b(j^{\prime})}({\rm Int } D_{2j^{\prime}})$.
\item If $a(j^{\prime})=b(j^{\prime})$, then $D_{2j^{\prime}-1}\bigcap D_{2j^{\prime}}$ is empty.
\item If we restrict the bundle $N(S_{a(j^{\prime})})$ over the sphere to $D_{2j^{\prime}-1}$ and the bundle $N(S_{b(j^{\prime})})$ over the sphere to $D_{2j^{\prime}}$, then the images of the total spaces of these resulting bundles by $c_{a(j^{\prime})}$ and $c_{b(j^{\prime})}$ agree as subsets in ${\mathbb{R}}^n$: the restrictions of the immersions to these spaces are embeddings.
\item The set of all crossings of the family $\{{c_j} {\mid}_{\partial N(S_j)}:\partial N(S_j)-S(S_j) \rightarrow P\}_{j=1}^l$ is the disjoint union of the $l^{\prime}$ corners of the subsets just before each of which is for $1 \leq j^{\prime} \leq l^{\prime}$. 
\end{enumerate}
\end{enumerate}
In this situation, the family $\{(S_j,N(S_j),c_j:N(S_j) \rightarrow P)\}_{j=1}^{l}$ is said to be a {\it normal system of submanifolds and subpolyhedra} compatible with $f$. 
\end{Def}
For regular neighborhoods in the smooth category, see \cite{hirsh} for example.
Via fundamental arguments on differential topology, in Definition \ref{def:6}, we have a pair of normal systems of submanifolds and subpolyhedra compatible with $f$ in the following definition as Definition \ref{def:4}.

\begin{Def}
\label{def:7}
The familiy $\{(S_j,N(S_j),c_j:N(S_j) \rightarrow P)\}_{j=1}^{l}$ is said to be a {\it wider normal system supporting} the normal system of submanifolds and subpolyhedra $\{(S_j,N^{\prime}(S_j),{c_j} {\mid}_{N^{\prime}(S_j)}:N^{\prime}(S_j) \rightarrow P)\}_{j=1}^{l}$ compatible with $f$.
\end{Def}

We easily have the following proposition. We can prove similarly and omit its proof.

\begin{Prop}
\label{prop:5}
we can define an {\rm ATSS} and the {\rm generating image} of the operation as Definition \ref{def:5} for Definition \ref{def:6} {\rm (}\ref{def:7}{\rm )}.
In the situation similar to Definition \ref{def:3} {\rm (}\ref{def:5}{\rm )}, we have a proposition similar to Proposition \ref{prop:1}.
\end{Prop}

\begin{Prop}
\label{prop:6}
\begin{enumerate}
\item We consider a situation as explained in Definition \ref{def:6} or Proposition \ref{prop:5}. Let the number of the spheres we take there be $l_j$ for each integer $1 \leq j \leq l$ and let $S_{j,j^{\prime}}$ denote the $j^{\prime}$-th sphere for $1 \leq j^{\prime} \leq l_j$ in $S_j$. Let $m>n \geq 1$ be integers. Let $n$ be even. Let $M$ be an $m$-dimensional closed and connected manifold.
For the normal system of submanifolds and subpolyhedra compatible with $f$, assume that $\dim S_j=\frac{n}{2}$ for any $j$. We also assume the relations $0<n-\frac{n}{2}=\frac{n}{2}<m-n,n<m-n+\frac{n}{2}=m-\frac{n}{2}<m$ and $m-n \neq n$.

In this situation, we can do an ATSS to $f$ such that the ${\bigcup}_{j=1}^l c_j(S_j)$ is the generating image of the operation and have a new map $f^{\prime}:M^{\prime} \rightarrow {\mathbb{R}}^n$ satisfying the following four properties.

\begin{itemize}
\item $H_{i}(M^{\prime};\mathbb{Z})$ is isomorphic to $H_{i}(M;\mathbb{Z}) \oplus {\mathbb{Z}}^l$ for $i=m-n,n$.
\item $H_{i}(M^{\prime};\mathbb{Z})$ is isomorphic to $H_{i}(M;\mathbb{Z}) \oplus {\mathbb{Z}}^{{\Sigma}_{j=1}^l l_j}$ for
$i=\frac{n}{2},m-\frac{n}{2}$.
\item $H_{i}(M^{\prime};\mathbb{Z})$ is isomorphic to $H_{i}(M;\mathbb{Z})$ for $i \neq \frac{n}{2},m-n,n,m-\frac{n}{2}$.
\item The total Stiefel-Whitney class of $M^{\prime}$ is $1 \in H^{0}(M^{\prime};\mathbb{Z}/2\mathbb{Z})$ if that of $M$ is so.
\end{itemize}

Furthermore, for the resulting Reeb space, we can construct the map satisfying the following two properties.
\begin{itemize}
\item $H_{i}(W_{f^{\prime}};\mathbb{Z})$ is isomorphic to $H_{i}(W_f;\mathbb{Z}) \oplus {\mathbb{Z}}^l$ for $i=\frac{n}{2},m-n,n,m-\frac{n}{2}$. Under suitable identifications of $H_{i}(M^{\prime};\mathbb{Z})$ with $H_{i}(M;\mathbb{Z}) \oplus {\mathbb{Z}}^l$ and $H_{i}(W_{f^{\prime}};\mathbb{Z})$ with $H_{i}(W_f;\mathbb{Z}) \oplus {\mathbb{Z}}^l$, the homomorphism ${q_{f^{\prime}}}_{\ast}$ between the homology groups induced from the quotient map $q_{f^{\prime}}$ maps an element $(0,p) \in H_{i}(M;\mathbb{Z}) \oplus {\mathbb{Z}}^l$ to $(0,p) \in H_{i}(W_f;\mathbb{Z}) \oplus {\mathbb{Z}}^l$. Under the identifications, ${q_{f^{\prime}}}_{\ast}$ maps an element $(p,0) \in H_{i}(M;\mathbb{Z}) \oplus {\mathbb{Z}}^l$ to $({q_f}_{\ast}(p),0) \in H_{i}(W_f;\mathbb{Z}) \oplus {\mathbb{Z}}^l${\rm :} the homomorphism ${q_f}_{\ast}$ is induced from $q_f:M \rightarrow W_f$ in a canonical way.
\item $H_{i}(W_{f^{\prime}};\mathbb{Z})$ is isomorphic to $H_{i}(W_f;\mathbb{Z})$ for $i \neq \frac{n}{2},m-n,n,m-\frac{n}{2}$.
\end{itemize}
\item In the case above, assume also that $n$ is divisible by $4$. 
Let $H:=(h_{j_1,j_2})$ be a ${\Sigma}_{j=1}^l l_j \times {\Sigma}_{j=1}^l l_j$ symmetric matrix values of whose components are integers satisfying $h_{j,j}=0$ for $1 \leq j \leq {\Sigma}_{j=1}^l l_j$ where $h_{j_1,j_2}$ denotes the $(j_1,j_2)$-th element.  
We also assume the following two.
\begin{enumerate}
\item For any pair $(c_{j_1} {\mid}_{S_{j_1}-S(S_{j_1})},c_{j_2} {\mid}_{S_{j_2}-S(S_{j_2})})$ of the immersions, there exist sufficiently many crossings. 
\item Identifications of homology groups used above are given.
\end{enumerate}
For $H_{\frac{n}{2}}(M;\mathbb{Z}) \oplus {\mathbb{Z}}^l$, let $e_{j,j^{\prime}}$ be the element represented as $(0,{e_{0,j,j^{\prime}}}) \in H_{\frac{n}{2}}(M;\mathbb{Z}) \oplus {\mathbb{Z}}^{{\Sigma}_{j=1}^l l_j}${\rm :} $e_{0,j,j^{\prime}}$ is the sequence of ${\Sigma}_{j=1}^l l_j$ integers the value of whose {\rm (}${\Sigma}_{j^{\prime \prime}=1}^{j-1} (l_{j^{\prime \prime}})+j^{\prime}${\rm )}-th component is $1$ and those of whose remaining components are all $0$.
For $H_{n}(M;\mathbb{Z}) \oplus {\mathbb{Z}}^l$, let ${e_j}^{\prime}$ be the element represented as $(0,{e_{0,j}}^{\prime}) \in H_{n}(M;\mathbb{Z}) \oplus {\mathbb{Z}}^l${\rm :} ${e_{0,j}}^{\prime}$ is the sequence of $l$ integers whose $j$-th component is $1$ and the other components are $0$.

In this situation, we can do an ATSS to $f$ such that the ${\bigcup}_{j=1}^l c_j(S_j)$ is the generating image of the operation and have a new map $f^{\prime}:M^{\prime} \rightarrow {\mathbb{R}}^n$ satisfying the following additional properties where ${q_{f^{\prime}}}^{\ast}:H^i(W_{f^{\prime}};\mathbb{Z}) \rightarrow H^i(M^{\prime};\mathbb{Z})$ is the homomorphism induced from the quotient map $q_{f^{\prime}}:M^{\prime} \rightarrow W_{f^{\prime}}$ for each $i$.
\begin{itemize}
\item For $e_{j,j^{\prime}}$ and ${e_j}^{\prime}$, we can define duals ${\rm D}(e_{j,j^{\prime}})$ and ${\rm D}({e_j}^{\prime})$ and by the statement above, we can define the duals ${\rm D}({q_{f^{\prime}}}_{\ast}(e_{j,j^{\prime}}))$ and ${\rm D}({q_{f^{\prime}}}_{\ast}({e_j}^{\prime}))$.
\item Two relations ${q_{f^{\prime}}}^{\ast}({\rm D}({q_{f^{\prime}}}_{\ast}(e_{j,j^{\prime}})))={\rm D}(e_{j,j^{\prime}})$ and ${q_{f^{\prime}}}^{\ast}({\rm D}({q_{f^{\prime}}}_{\ast}({e_j}^{\prime})))={\rm D}({e_j}^{\prime})$ hold.
\item The product of ${\rm D}(e_{j_1,{j_1}^{\prime}})$ and ${\rm D}(e_{j_2,{j_2}^{\prime}})$ is\\
$h_{{\Sigma}_{j^{\prime \prime}=1}^{j_1-1} (l_{j^{\prime \prime}})+{j_1}^{\prime},{\Sigma}_{j^{\prime \prime}=1}^{j_2-1} (l_{j^{\prime \prime}})+{j_2}^{\prime}}{\rm D}({e_{j_1}}^{\prime})+h_{{\Sigma}_{j^{\prime \prime}=1}^{j_2-1} (l_{j^{\prime \prime}})+{j_2}^{\prime},{\Sigma}_{j^{\prime \prime}=1}^{j_1-1} (l_{j^{\prime \prime}})+{j_1}^{\prime}}{\rm D}({e_{j_2}}^{\prime})$.
\end{itemize}
\end{enumerate}
\end{Prop}
\begin{proof}[A sketch of a proof.]
We can discuss the first statement as the case of Proposition \ref{prop:2}.
In Definition \ref{def:6} or Proposition \ref{prop:5}, we can take a regular neighborhood $N(S_{j,j^{\prime}}) \subset N(S_j)$ of $S_{j,j^{\prime}} \subset S_j$. We can take this regular neighborhood as the total space of a linear bundle over $S_{j,j^{\prime}}$ whose fiber is diffeomorphic to a unit disk. Furthermore, we can do this so that $S_{j,j^{\prime}}$ is regarded as the image of the section taking the origin in the fiber at each point in $S_{j,j^{\prime}}$. Moreover, $N(S_j)$ can be regarded as a manifold represented as a boundary connected sum of the manifolds $N(S_{j,j^{\prime}})$.

Consider a fiber $D_{{\rm F},j}$ of a regular neighborhood $N(S_{j,j^{\prime}}) \subset N(S_j)$ of $S_{j,j^{\prime}} \subset S_j$ such that $c_j(D_{{\rm F},j})$ contains no crossings of the family $\{c_j {\mid}_{\partial N(S_{j,j^{\prime}})}\}$ of the immersions,
set $Q_{{\rm F},j,j^{\prime}}:=f^{-1}(c_j(D_{{\rm F},j,j^{\prime}}))$ and consider the restriction of $f$ to $Q_{{\rm F},j,j^{\prime}}$. 
Note that $Q_{{\rm F},j,j^{\prime}}$ is a compact manifold diffeomorphic to the manifold $E_D$ in the sketch of the proof of Proposition \ref{prop:2}.

We can construct a desired map $f^{\prime}:M^{\prime} \rightarrow {\mathbb{R}}^n$ on a suitable manifold $M^{\prime}$.

We explain a generator of the $j$-th summand of ${\mathbb{Z}}^l$ in $H_{i}(M;\mathbb{Z}) \oplus {\mathbb{Z}}^l$, isomorphic to $H_{i}(M^{\prime};\mathbb{Z})$, for $i=m-n,n$.

For $i=m-n$, it is represented by a connected component of the preimage of a regular value, diffeomorphic to $S^{m-n}$.

For $i=n$, it is represented by a smooth submanifold represented by a double of $N(S_j)$ mapped to the image $c_j(N(S_j))$ by $f$.

For a generator of the (${\Sigma}_{j^{\prime \prime}=1}^{j-1} (l_{j^{\prime \prime}})+j^{\prime}$)-th summand of ${\mathbb{Z}}^{{\Sigma}_{j^{\prime \prime}=1}^l l_{j^{\prime \prime}}}$ in $H_{i}(M;\mathbb{Z}) \oplus {\mathbb{Z}}^{{\Sigma}_{j^{\prime}=1}^l l_{j^{\prime}}}$, isomorphic to $H_{i}(M^{\prime};\mathbb{Z})$ ($i=\frac{n}{2},m-\frac{n}{2}$). we can discuss as the sketch of the proof of Proposition \ref{prop:2}.
We can see that the remaining properties of the first statement hold by the way of this construction.

We can show the second statement as Proposition \ref{prop:3}. The assumption that we use a normal system of submanifolds and subpolyhedra instead of a normal system of submanifolds does not change the argument essentially.

\end{proof}
We can also show the following theorem.

\begin{Thm}
\label{thm:6}
In the situation of Theorem \ref{thm:5}, consider an arbitrary integral cohomology ring $C_R$ obtained in the following way starting from $H^{\ast}(M;\mathbb{Z})$. 
\begin{enumerate}
\item Choose finitely many elements of basis $\{{b_{j,2}}^{\ast}\}$ and identify them.
\item Obtain the quotient integral cohomology ring in a canonical way.
\item We do a finite iteration of the following procedure: choose finitely many elements of the basis obtained in a canonical way and identify them and obtain the quotient in a canonical way
\end{enumerate}

Thus there exists a $7$-dimensional closed and simply-connected spin manifold $M_{C,R}$ whose cohomology ring is isomorphic to $C_R$ and a fold map $f:M_{C,R} \rightarrow {\mathbb{R}^4}$ satisfying the following properties
\begin{enumerate}
\item The 3rd and the 5th Stiefel-Whitney classes of $M$ vanish.
\item Let $d$ be the isomorphism $d:H_4(M;\mathbb{Z}) \rightarrow H^4(M;\mathbb{Z})$ defined by corresponding the duals. The first Pontryagin class of $M$ is regarded as $4d \circ {\phi}_4(p) \in H^4(M;\mathbb{Z})$. The 4th Stiefel-Whitney class of $M$ vanishes.
\item The immersion $f {\mid}_{S(f)}$ is normal.
\item The index of each singular point is $0$ or $1$.
\item Preimages of regular values in the image of $f$ are diffeomorphic to $S^3$, $S^3 \sqcup S^3$ or $S^3 \sqcup S^3 \sqcup S^3$.
\end{enumerate}
\end{Thm}
To prove, we consider an ATSS in Proposition \ref{prop:5} and apply Proposition \ref{prop:6} instead in the situation of the proof of Theorem \ref{thm:5}. 
The five additional properties can be shown similarly. Rigorous proofs are left to readers.

The final remark is as follows.
\begin{Rem}
\label{rem:2}
Most maps in Theorems \ref{thm:1}--\ref{thm:6} can be constructed as fold maps such that we can represent as the compositions of the smooth embeddings into the one-dimensional higher Euclidean spaces with the canonical projections. More precisely, if the Pontryagin classes of higher degrees vanish, then we can construct such maps. It has been shown as a corollary in \cite{eliashberg2} (\cite{eliashberg}) that closed manifolds admitting smooth embeddings into the one-dimensional higher Euclidean spaces admit fold maps into arbitrary Euclidean spaces whose dimensions are smaller than or equal to those of the manifolds.
\end{Rem}


\begin{thebibliography}{30}
\bibitem{eellskuiper} J. J. Eells and N. H. Kuiper, \textsl{An invariant for certain smooth manifolds}, Ann. Mat. Pura Appl. 60 (1962), 93--110.
\bibitem{eliashberg} Y. Eliashberg, \textsl{On singularities of folding type}, Math. USSR Izv. 4 (1970). 1119--1134.
\bibitem{eliashberg2} Y. Eliashberg, \textsl{Surgery of singularities of smooth mappings}, Math. USSR Izv. 6 (1972). 1302--1326.
\bibitem{hirsch} Morris W. Hirsch, \textsl{Smooth regulart neighborhoods}, Annals of Mathematics Second Series, Vol. 76, No. 3 (Nov., 1962), pp. 524--530.
\bibitem{kitazawa} N. Kitazawa, \textsl{On round fold maps} (in Japanese), RIMS Kokyuroku Bessatsu B38 (2013), 45--59.
\bibitem{kitazawa2} N. Kitazawa, \textsl{On manifolds admitting fold maps with singular value sets of concentric spheres}, Doctoral Dissertation, Tokyo Institute of Technology (2014).
\bibitem{kitazawa3} N. Kitazawa, \textsl{Fold maps with singular value sets of concentric spheres}, Hokkaido Mathematical Journal Vol.43, No.3 (2014), 327--359.
\bibitem{kitazawa4} N. Kitazawa, \textsl{Round fold maps and the topologies and the differentiable structures of manifolds admitting explicit ones}, submitted to a refereed journal, arXiv:1304.0618 (the title has changed).
\bibitem{kitazawa5} N. Kitazawa, \textsl{Constructing fold maps by surgery operations and homological information of their Reeb spaces}, submitted to a refereed journal, arxiv:1508.05630 (the title has been changed).
\bibitem{kitazawa6} N. Kitazawa, \textsl{Notes on fold maps obtained by surgery operations and algebraic information of their Reeb spaces}, submitted to a refereed journal, arxiv:1811.04080.
\bibitem{kitazawa7} N. Kitazawa, \textsl{Notes on explicit smooth maps on 7-dimensional manifolds into the 4-dimensional Euclidean space}, submitted to a refereed journal, arxiv:1911.11274.
\bibitem{kitazawa8} N. Kitazawa, \textsl{Surgery operations to fold maps to construct fold maps whose restrictions to the singular sets may not be embeddings}, arxiv:2003.04147v8.
\bibitem{kitazawa9} N. Kitazawa, \textsl{Surgery operations to fold maps to increase connected components of singular sets by two}, arxiv:2004.03583.
\bibitem{kobayashi} M. Kobayashi, \textsl{Stable mappings with trivial monodromies and application to inactive log-transformations}, RIMS Kokyuroku. 815 (1992), 47--53.
\bibitem{kobayashi2} M. Kobayashi, \textsl{Bubbling surgery on a smooth map}, preprint.
\bibitem{kobayashisaeki} M. Kobayashi and O. Saeki, \textsl{Simplifying stable mappings into the plane from a global viewpoint}, Trans. Amer. Math. Soc. 348 (1996), 2607--2636.  
\bibitem{kreck} M. Kreck,  \textsl{On the classification of $1$-connected $7$-manifolds with torsion free second homology}, to appear in the Journal of Topology, arxiv:1805.02391.
\bibitem{milnor} J. Milnor, \textsl{On manifolds homeomorphic to the $7$-sphere}, Ann. of Math. (2) 64 (1956), 399--405.
\bibitem{milnorstasheff} J. Milnor and J. Stasheff, \textsl{Characteristic classes}, Annals of Mathematics Studies, No. 76. Princeton, N. J; Princeton University Press (1974).
\bibitem{reeb} G. Reeb, \textsl{Sur les points singuliers d\`{u}ne forme de Pfaff completement integrable ou d’une fonction numerique}, -C. R. A. S. Paris 222 (1946), 847--849. 
\bibitem{saeki} O. Saeki, \textsl{Topology of special generic maps of manifolds into Euclidean spaces}, Topology Appl. 49 (1993), 265--293.
\bibitem{saeki2} O. Saeki, \textsl{Topology of special generic maps into $\mathbb{R}^3$}, Workshop on Real and Complex Singularities (Sao Carlos, 1992), Mat. Contemp. 5 (1993), 161--186.
\bibitem{saekisakuma} O. Saeki and K. Sakuma, \textsl{On special generic maps into ${\mathbb{R}}^3$}, Pacific J. Math. 184 (1998), 175--193.
\bibitem{saekisuzuoka} O. Saeki and K. Suzuoka, \textsl{Generic smooth maps with sphere fibers} J. Math. Soc. Japan Volume 57, Number 3 (2005), 881--902.
\bibitem{sakuma} K. Sakuma, \textsl{On special generic maps of simply connected $2n$-manifolds into ${\mathbb{R}}^3$},
\bibitem{sakuma2} K. Sakuma, \textsl{On the topology of simple fold maps}, Tokyo J. of Math. Volume 17, Number 1 (1994), 21--32.
\bibitem{shiota} M. Shiota, \textsl{Thom's conjecture on triangulations of maps}, Topology 39 (2000), 383--399.
\bibitem{steenrod} N. Steenrod, \textsl{The topology of fibre bundles}, Princeton University Press (1951). 
\bibitem{thom} R. Thom, \textsl{Les singularites des applications differentiables}, Ann. Inst. Fourier (Grenoble) 6 (1955-56), 43--87.
\bibitem{wang} X. Wang, \textsl{On the classification of certain $1$-connected $7$-manifolds and related problems}, arXiv:1810.08474.
\bibitem{whitney} H. Whitney, \textsl{On singularities of mappings of Euclidean spaces: I, mappings of the plane into the plane}, Ann. of Math. 62 (1955), 374--410.
\bibitem{wrazidlo} D. J. Wrazidlo, \textsl{Standard special generic maps of homotopy spheres into Eucidean spaces}, Topology Appl. 234 (2018), 348--358, arxiv:1707.08646.
\end{thebibliography}
\end{document}